\documentclass[11pt]{amsart}
\usepackage[a4paper,margin=0.9in]{geometry}
\usepackage{comment}
\usepackage[colorlinks,citecolor=magenta,linkcolor=black]{hyperref}
\pdfpagewidth=\paperwidth \pdfpageheight=\paperheight
\usepackage{amsfonts,amssymb,amsthm,amsmath,tabu,url}
\usepackage{pgf}
 \usepackage{array}
 \usepackage{tikz-cd}
 \usepackage{pstricks}
 \usepackage{pstricks-add}
 \usepackage{pgf,tikz}
 \usetikzlibrary{automata}
 \usetikzlibrary{arrows}
 \usepackage{indentfirst}
\usepackage{tabularx} 

\DeclareMathOperator{\Sing}{Sing}
\DeclareMathOperator{\Irr}{Irr}

\def\cC{\mathcal{C}}

\def\cP{{\mathcal P}}
\def\cN{{\mathcal N}}

\def\CC{\mathbb{C}}

\def\PP{\mathbb{P}}

\theoremstyle{plain}
\newtheorem{thm}{Theorem}
\newtheorem{theorem}[thm]{Theorem}

\newtheorem{lemma}[thm]{Lemma}
\newtheorem{proposition}[thm]{Proposition}
\newtheorem{corollary}[thm]{Corollary}

\theoremstyle{definition}

\newtheorem{remark}[thm]{Remark}

\newtheorem{question}[thm]{Question}

\newtheorem{thevarthm}[thm]{\varthmname}

\newenvironment{varthm*}[1]{\trivlist\item[]{\bf #1.}\it}{\endtrivlist}

\def\subclassname{{\bfseries Mathematics Subject Classification
(2020)}\enspace}
\def\subclass#1{\par\addvspace\medskipamount{\rightskip=0pt plus1cm
\def\and{\ifhmode\unskip\nobreak\fi\ $\cdot$
}\noindent\subclassname\ignorespaces#1\par}}

\begin{document}
\title{Modular inequalities and Alexander polynomials of pencil type conic-line arrangements}

\author[Anca~M\u acinic]{Anca~M\u acinic$^{*}$}

\thanks{$^{*}$ Partially supported by the project "Singularities and Applications'' - CF 132/31.07.2023 funded by the European Union - NextGenerationEU - through Romania's National Recovery and Resilience Plan.}
\address{Simion Stoilow Institute of Mathematics, 21 Calea Grivitei Street, 010702 Bucharest, Romania}
\email{Anca.Macinic@imar.ro}

\date{\today}

 
\maketitle

\thispagestyle{empty}
\begin{abstract}
We use recent results, among which modular inequalities for curves, to determine the Alexander polynomials for some classes of pencil-type conic-line arrangements. For these classes of curves we prove that the Alexander polynomial is (at least partially) combinatorial. To this end, we exemplify new  techniques that are suitable for broader use, lending themselves to more general classes of curves. 
\end{abstract}

\section{Introduction}

Seen as a natural generalization of arrangements of projective lines in the field of projective plane curves, {\it conic-line arrangements} (or, equivalently, {\it conic-line curves}) are defined as curves whose irreducible components are lines and smooth conics. There are numerous recent articles that explore algebraic, geometric, topological and combinatorial properties of conic-line arrangements, see for instance \cite{BMR, DLPU, DPS, M, MP1, Pk, ST}. 

Given a complex projective plane curve $\cC$ defined by a homogeneous degree $d$ polynomial $f_{\cC} \in \CC[x,y,z]$, we consider the Milnor fiber associated to $\cC$, i.e. the hypersurface in $\CC^3$ defined by $$\mathcal{F}(\cC): f_{\cC}=1.$$ The multiplication by a primitive root of unity of order $d$ defines an action on $\mathcal{F}(\cC)$. This induces an action on the homology  group $H_1(\mathcal{F}(\cC), \mathbb{Q})$, called algebraic monodromy, which makes $H_1(\mathcal{F}(\cC), \mathbb{Q})$ a $\mathbb{Q}[t^{\pm 1}]$-module, where $t$ acts by the algebraic monodromy.
 Since the monodromy is of finite order $d$, $H_1(\mathcal{F}(\cC), \mathbb{Q})$ is a torsion module.
The Alexander polynomial of a curve, $\Delta_{\cC}(t)$, is defined as the characteristic polynomial of the monodromy action on $H_1(\mathcal{F}(\cC), \mathbb{Q})$ and, in light of the above discussion, it satisfies the formula
\begin{equation}
\label{eq:Alexander_poly}
\Delta_{\cC}(t) = \Pi_{k|d}\phi_k^{b_k},
\end{equation}
where $\phi_k$ is the $k$-th cyclotomic polynomial.

Computing the monodromy action on the homology of the Milnor fiber is a very difficult problem, even in the case of degree $1$ homology  (meaning,  computing the Alexander polynomial), and even if we narrow down the problem to the case of arrangements of projective lines. To name just a few papers on the subject, see \cite{BDS, CL2, CS, D0, DIM, DS, Li2, Li3, Li4, MPP}.  The study of the Alexander polynomial of a curve is also relevant in connection to the fundamental group of the complement of the curve in the complex projective plane, see for instance \cite{ACT}.

In \cite{CMc}, based on a result of Papadima-Suciu from \cite{PS2},  the authors introduce modular inequalities for (complements of) complex projective plane curves, inequalities that give upper bounds for the multiplicities of roots of the associated Alexander polynomial, provided that those roots are powers of primes. In the same paper  \cite{CMc}, lower bounds are given for the  multiplicities of roots of the associated Alexander polynomial, for quasi fiber-type curves. 
In this note we  will use these results to compute  Alexander polynomials for pencil-type conic-line arrangements, but the techniques employed for this purpose are likely to be of use in the computation of  the Alexander polynomial of other general classes of curves.

 By {\it pencil of conic-line arrangements}, or, equivalently, {\it pencil of conic-line curves}, we mean a primitive pencil $\cP$ of degree $d \geq 3$ curves having at least three reducible and reduced members whose irreducible components are lines and smooth conics. We will call these members {\it conic-line fibers}. 
    A {\it pencil-type conic-line arrangement} is a reduced curve that can be realized as an union of  conic-line fibers  of a  pencil of conic-line arrangements.
 
An intensely studied and still open problem concerning arrangements of hyperplanes, related to the Milnor fiber monodromy problem,  is whether the Alexander polynomial is determined by the combinatorics of  the intersection lattice of the arrangement, i.e. if it is {\it combinatorial}.  A comprehensive presentation of the problem and various approaches to solve it can be found in \cite[Subsection 1.8]{PS3}.
The general case of this problem can be reduced to the case of arrangements of projective lines. 
 On the other hand, we know that this conjecture does not hold for plane projective curves in general.  In the context of curves, by  {\it combinatorial} we mean determined by the weak combinatorial type of the curve, see Section \S\ref{sec:preliminaries} for the definition. Counterexamples of the combinatorial determination of the Alexander polynomial for curves can be found among {\it Zariski pairs} 
  (briefly, Zariski pairs are pairs of curves with the same weak combinatorial type but different embedded in $\PP^2$ topology, see for instance \cite{ACT} for a detailed exposition on the subject).
One may consider for instance the pair of irreducible sextics that Zariski introduced in \cite{Zar}.
The two curves are combinatorially equivalent but it turns out they have distinct Alexander polynomials. 
Examples involving reducible curves that are combinatorially equivalent but have distinct Alexander polynomials 
 have bee constructed as well, see for instance Artal Bartolo's  example from \cite{AC} of a pair of degree $6$ curves, each consisting of a union of a smooth cubic and three lines, which is also a Zariski pair.  Or, more generally, we refer to the Zariski pairs of sextics from Oka's papers  \cite{Oka, OkaII}, emerging from sextics classification results.

However, we were able to detect a distinctive phenomenon that happens for a subclass of pencil-type conic-line arrangements. 
Specifically, for this subclass, for which we give formulas for the Alexander polynomials, in Theorems \ref{thm:Alex_poly_formulas} and \ref{thm:m=4_intro}, we prove that the multiplicities of the roots of the Alexander polynomials that we compute are {\it combinatorial}, see Theorem \ref{thm:main}. These results are reminiscent of similar results in the theory of arrangements of hyperplanes, see \cite{MP, PS3}.

In \cite{CL} Cogolludo-Libgober prove that in a pencil-type conic-line arrangement having only ordinary singular points as base points and which is realized as the union of $m$ conic-line fibers of a conic-line pencil of curves of degree $d \geq 3$, necessarily $m \leq 6$. We should point out that this statement, recalled in Theorem \ref{thm:CL_classification}, involves two cases different in nature, see \cite[Proposition 4.2]{CL} for the case $d \geq 4$ and \cite[Corollary 4.4]{CL} \& \cite{R} for the case $d=3$. 
The bound on $m$ in the case $d=3$ is sharp, as illustrated by an example of Ruppert (see \cite{R}) of a a pencil-type conic-line arrangement which is the union of $6$ reducible conic-line fibers, each fiber being the union of a line and a smooth conic.  The bound on $m$ in the case $d=4$ is also known to be sharp, 
 see \cite[Proposition 4.5]{CL}.

We study here the class of pencil-type conic-line arrangements with transversal base points coming from pencils of degree $3$  curves. 

\begin{theorem}
\label{thm:main}
Let $\mathcal{C}$ be a pencil-type conic-line arrangement arising from a pencil of degree $3$ curves such that the singularities of $\cC$ at the base points of the pencil are transversal.
Then the multiplicity of primitive roots of unity of odd order as roots of the Alexander polynomial $\Delta_{\cC}(t)$ is combinatorial. Moreover, if there are at most two non-ordinary singular points in $\Sing(\cC)$, then the multiplicity of primitive roots of unity of order powers of $2$ as roots of the Alexander polynomial $\Delta_{\cC}(t)$ is combinatorial as well.
\end{theorem}

Furthermore, we determine the Alexander polynomial of the curves $\cC$ described in the above theorem.

\begin{theorem}
\label{thm:Alex_poly_formulas}
Let $\mathcal{C}$ be a pencil-type conic-line arrangement which is realized as the union of $m \geq 3$ conic-line fibers of a conic-line pencil of degree $3$ curves. Assume that the singularities of $\mathcal{C}$ at the base points of the pencil are ordinary.
Denote $s:=|\Irr(\cC)|$. Then:
\begin{enumerate}
\item If $m=3$, then $\Delta_{\cC}(t) = (t-1)^{s-1}(t^2+t+1)$, except if $\cC$ is the Ceva arrangement, in which case  $\Delta_{\cC}(t) = (t-1)^{s-1}(t^2+t+1)^2$.
\item If $m=4$, then $\Delta_{\cC}(t) = (t-1)^{s-1}(t+1)^{b_2}(t^2+1)^{b_4}$ and $b_2, b_4 \geq 2$. 
\item If $m=5$, then $\Delta_{\cC}(t) = (t-1)^{s-1}(t^4+t^3+t^2+t+1)^3$.
\item If $m=6$, then $\Delta_{\cC}(t) = (t-1)^{s-1}(t+1)^{b_2}(t^2+t+1)^4(t^2-t+1)^{b_{6}}$ and $b_2, b_6 \geq 4$.
\end{enumerate}
Moreover, if there are at most two non-ordinary singularities in $\Sing(\cC)$, then 
\begin{enumerate}
\item If $m=4$, then $b_2 = b_4 =2$. 
\item If $m=6$, then $b_2 =4$. 
\end{enumerate}
\end{theorem}

When $m=4$ we can ease the restrictions on the number of non-ordinary singularities. We prove the following result, see  Theorem \ref{thm:m=4_Alex_poly} for a rephrasing.

\begin{theorem}
\label{thm:m=4_intro}
Let $\mathcal{C}$ be a pencil-type conic-line arrangement as in Theorem \ref{thm:Alex_poly_formulas} and assume $m=4$. 
If either one of the two conditions hold:
\begin{enumerate}
\item There exist at most three non-ordinary singularities in $\Sing(\cC)$,
\item  There exist a base point $P$ of the pencil such that an odd number of lines from $\Irr(\cC)$ pass through $P$,
\end{enumerate}
then  $\Delta_{\cC}(t) = (t-1)^{s-1}(t+1)^2(t^2+1)^2$.
\end{theorem}

We give in  Subsection \S\ref{ss:ACC_example} an abstract curve configuration that illustrates a situation when none of the two conditions from  Theorem \ref{thm:m=4_intro} are met. A curve that realizes this configuration would potentially give an example of Alexander polynomial with the root multiplicities $b_2, b_4$ satisfying  $2 < b_2, b_4$. In any case, we are able to show that, for this combinatorial type,  $b_2, b_4 \leq 3$. 
Moreover, we show that this configuration is the unique weak combinatorial type of a  pencil-type conic-line arrangement $\cC$ which is the union of the $m=4$ conic-line fibers of a pencil of degree $3$ curves such that conditions (1) and (2) do not hold, see Theorem \ref{thm:uniqueACC}. We end Section \S\ref{sec:m=4} with a couple of open questions regarding this abstract curve configuration.\\

\section*{Acknowledgments}
I would like to thank Jose Ignacio Cogolludo-Agust\'in for his remarks and clarifications regarding the results from \cite{CL}. I would also like to thank Piotr Pokora for his remarks leading to an improved exposition.

\section{Preliminaries}
\label{sec:preliminaries}
Let $\cC = \cC_1 \cup \dots \cup \cC_r$ be the decomposition into irreducible components of the reduced plane projective curve $\cC$ and denote $d_i:= \deg(\cC_i)$.

The data that consists of:
\begin{itemize}
\item $\Sing(\cC)$, the set of  singularities of $\cC$;
\item the set of labels $ \mathbf r := \{1, \dots, r\}$ of the irreducible components;
\item for each $P \in \Sing(\cC)$ the set of local branches $\Delta_P$ of $\cC$ at $P$;
\item the map $\phi_P$ that assigns to each branch at $P$ the label of the irreducible component it belongs to;
\item   the map that assigns to each pair of branches (from distinct components)  at $P$ the intersection number of the two;
\end{itemize} 
 is called the {\it weak combinatorial type} of the curve $\cC$, see \cite{C}. We will denote it by $W_{\cC}$.
 The weak combinatorial type can be seen as playing the same role for curves as the notion of intersection lattice for arrangements of hyperplanes. Furthermore,  the notion of {\it abstract curve combinatorics} was developed in \cite{BC} to capture formally the defining characteristics of a weak combinatorial type, see \cite{BC} for a precise definition. Taking further the parallelism to the theory of hyperplane arrangements, an abstract curve combinatorics is reminiscent of the notion of matroid.
 We will get back to this in Subsection \S\ref{ss:ACC_example}.
 
 Keeping in mind the analogy to arrangements of hyperplanes, we will call an invariant or property of a curve $\cC$ {\it combinatorial} if it depends only on the weak combinatorial type of $\cC$. For instance, Cogolludo-Matei prove that the cohomology algebra of the complement of a  plane curve is combinatorial, building on a similar result by Cogolludo for rational curves, see \cite{C, CM}.
 This mirrors Orlik-Solomon's emblematic result on the combinatoriality of the cohomology of the complement of an arrangement of hyperplanes in a complex vector space.
 
 To the weak combinatorial type of a curve $\mathcal{C}$ one associates a graded algebra, $(A^*_{W_{\mathcal{C}}}, \wedge)$ over $\Bbbk=\mathbb{Z}_p$, $p$ prime, defined as follows  (see \cite{CMc}). The $\Bbbk$-module structure is given by
\begin{itemize}
\item $A_{W_{\mathcal{C}}}^0:=\Bbbk$,
\item
$A_{W_{\mathcal{C}}}^1:=\bigoplus_{i\in \mathbf r} \sigma_i\Bbbk$,
\item
$A_{W_{\mathcal{C}}}^2:=A_{W_{\mathcal{C}}}^{2,0}\oplus A_{W_{\mathcal{C}}}^{2,\infty}$, where
$$
A_{W_{\mathcal{C}}}^{2,\infty}:=\langle \bar\psi^i_\infty : i\in\mathbf{r}\rangle_\Bbbk/\langle \bar\psi^i_\infty: p\nmid d_i \rangle_\Bbbk,\quad
A_{W_{\mathcal{C}}}^{2,0}:=\bigoplus_{P\in S} \frac{A_P\wedge A_P}{I_P},\quad
A_P:=\langle \psi_P^\delta : \delta\in \Delta_P\rangle_\Bbbk,
$$
$$
I_P:=\{
\psi_P^{\delta_1}\wedge \psi_P^{\delta_2}+\psi_P^{\delta_2}\wedge \psi_P^{\delta_3}+
\psi_P^{\delta_3}\wedge \psi_P^{\delta_1}\}+
\{\psi_P^{\delta_1}\wedge \psi_P^{\delta_2}\mid \#\phi_P(\Delta_P)=1\}\subset A_P\wedge A_P.
$$
\item $A_{W_{\mathcal{C}}}^i=0$ ($i\geq 3$).
\end{itemize}
while the product, defined on the degree $1$ generators as described below, makes $A^*_{W_{\mathcal{C}}}$ an algebra
$$
\sigma_i \wedge \sigma_j:=
\sum_{\substack{P \in S,\\  \delta_1\in\phi_P^{-1}(i),\\ \delta_2\in\phi_P^{-1}(j)}}
\mu_P(\delta_1,\delta_2)\ \psi_P^{\delta_1}\wedge \psi_P^{\delta_2}\wedge
+d_j\bar\psi_\infty^i-d_i\bar\psi_\infty^j.
$$

  The multiplication by an arbitrary $1$-form  $\omega \in A^1_{W_{\mathcal{C}}}$ defines a complex structure $(A^*_{W_{\mathcal{C}}}, \wedge \omega)$ on the algebra $A^*_{W_{\mathcal{C}}}$, called the {\it Combinatorial Aomoto complex}. Accordingly we define the $k$-th resonance varieties associated to these complexes:
  $$
 \mathcal{R}_k(W_{\mathcal{C}}, \Bbbk) := \{ \omega \in  A^1_{W_{\mathcal{C}}} | \dim_{\Bbbk}(H^1(A^*_{W_{\mathcal{C}}}, \wedge \omega) \geq k) \}
$$

\noindent With this, let us recall a result from \cite{CMc} that gives an upper bound for the exponents $b_k$ from \eqref{eq:Alexander_poly}.

\begin{theorem}
\label{thm:upper_bound}
Let $k=p^s, s \geq 1$ and $\omega_1:=\sum_{i \in \mathbf r} \sigma_i$. Then 
 $b_k \leq  \dim_{\Bbbk} H^1(A^*_{W_{\mathcal{C}}}, \wedge \omega_1).$
\end{theorem}

Notice that the right hand side of the inequality in Theorem \ref{thm:upper_bound} is combinatorial. Moreover, we have a result that allows the computation of this combinatorial invariant (compare to \cite[Theorem 2.5]{F}, its analogue for arrangements of hyperplanes).

\begin{proposition}(\cite{CMc}) 
\label{prop:beta_system}
Let $\omega = \sum_{i \in \mathbf r} a_i\sigma_i\in A^1_{W_{\mathcal{C}}}$. Then $\omega \wedge \omega_1=0$ if and only if,  for any $P \in \Sing(\mathcal{C})$ and any $\delta \in \Delta_P, \; \phi_P(\delta)=i$, we have
\begin{equation}
\label{eq:system}
\left| \array{cc} a_i & 1\\
\sum_k \mu_P(\delta,\mathcal{C}_k)a_k & \sum_k \mu_P(\delta,\mathcal{C}_k) \endarray \right|=0.
\end{equation}
and, for any $i \in \mathbf r$ such that $p \mid d_i$, we have
\begin{equation}
\label{eq:p| d_i}
\left| \array{cc} a_i & 1\\
\sum_k d_k a_k & \sum_k d_k \endarray \right|=0.
\end{equation}
\end{proposition}

In particular, for ordinary singularities the equations satisfied by the coefficients $a_i$ of the resonant $1$-form $\omega$ 
 can be stated in simpler terms.

\begin{proposition}(\cite{CMc})
\label{prop:ordinary_sing_system}
Let $\omega = \sum _{i \in \mathbf r} a_i\sigma_i\in A^1_{W_{\mathcal{C}}}$ such that $\omega \wedge \omega_1=0$ and $P \in \Sing(\mathcal{C})$ an ordinary singularity such that the number of branches at $P$ coincides to the number of irreducible components $ \mathcal{C}_{i_1}, \cdots,  \mathcal{C}_{i_s} $ of $\mathcal{C}$ intersecting at $P$. 
Then either 
\begin{equation}
\label{eq:equalities}
a_{i_1} = \dots = a_{i_s},  \textrm{ if } s \not\equiv 0 \;  \textrm{mod p},
\end{equation}
or 
\begin{equation}
\label{eq:sum}
a_{i_1} + \dots +a_{i_s}=0\in\Bbbk \textrm{ if } s \equiv 0 \;  \textrm{mod p}.
\end{equation}
\end{proposition}

For convenience, given a curve $\mathcal{C}$,  we will denote by $\beta_p$ the  combinatorial invariant of the curve
\begin{equation}
\label{eq:beta_p_def}
\beta_p= \beta_p(\mathcal{C}) =\dim_{\Bbbk} H^1(A^*_{W_{\mathcal{C}}}, \wedge \omega_1). 
\end{equation}
With this notation, if  $k=p^s, \; s \geq 1$, the inequality in Theorem \ref{thm:upper_bound} becomes
\begin{equation}
\label{eq:beta_ineq}
b_k \leq \beta_p.
\end{equation}


\begin{remark}
\label{rem:beta_system}
Proposition  \ref{prop:beta_system} amounts to the existence of a linear system  {\em S} of equations with $\mathbb{Z}_p$ coefficients whose space of solutions is of dimension $\beta_p(\cC)+1$. More precisely,  {\em S}  consists  of the combined sets of equations \eqref{eq:system} and \eqref{eq:p| d_i}, in the variables $a_i, i \in \overline{1, |\Irr(\cC)|}$, over $\mathbb{Z}_p$. 
\end{remark}

\cite[Theorem 37]{CMc} gives a lower bound for certain exponents $b_k$ from \eqref{eq:Alexander_poly}, for curves admitting quasi-fiber type structures. As a consequence, we get the following (see \cite[Corollary 38]{CMc} for a more general version of the next result, phrased in terms of twisted Alexander polynomials).

\begin{theorem}
\label{thm:lower_bound}
Let $\cC$ be a reduced fiber-type curve that is the union of  $s \geq 3$ members of a pencil. Then, for any $k \mid s$, the multiplicity of a primitive root of unity of order $k$ as a root of the Alexander polynomial $\Delta_{\cC}(t)$ is at least $s-2$.
\end{theorem}

Finally, we state a classic result on the Alexander polynomial of a curve. To this end, recall that the local Alexander polynomial of a singularity $P$ of a curve $\cC$ is the characteristic polynomial of the corresponding local monodromy operator.

\begin{theorem} (\cite{Li1, D})
\label{thm:global|local}
The Alexander polynomial of a curve $\cC$ divides the product of the local Alexander polynomials  $\Delta_{(\cC, P)}(t)$, product taken over all singularities $P \in \Sing(\cC)$. Moreover, for any irreducible component $C$ of $\cC$, the Alexander polynomial of  $\cC$ divides the product of the local Alexander polynomials  $\Delta_{(\cC, P)}(t)$, product taken over all singularities $P \in \Sing(\cC)$ situated on the component $C$.
\end{theorem}

We will use the above result in the next section as a way to filter the possible roots of the Alexander polynomials of the  pencil-type conic-line curves from Theorem \ref{thm:main}. Since these curves have as singularities only transversal or type $A_3$ singularities, we need to recall the formulas of of the local Alexander polynomials associated to these types of singularities. Denote by $X_n$ an ordinary singularity with $n$ smooth branches.

\begin{lemma}
\label{lem:local_Alexander_poly}
\begin{equation}
\label{eq:A3_local_Apoly}
\Delta_{A_3}(t) = (t-1)(t^2+1)
\end{equation}
\begin{equation}
\label{eq:Xn_local_Apoly}
\Delta_{X_n}(t) = (t-1)(t^n-1)^{n-2}
\end{equation}
\end{lemma}

\begin{proof}
One can compute the monodromy operator of the singularities (at $(0,0)$)  $x^2+y^4=0$ and $x^m-y^m=0$ directly or via the Alexander polynomials of the links associated to the respective singularities,  $x^2+y^4=0$ and $x^m-y^m=0$,  
see for instance \cite{Mil, Br}. 
\end{proof}


\begin{lemma}
\label{lem:b_q_trivial}
 Let $m \geq 3$ be the number of conic-line  fibers of a pencil of degree $3$ curves and let $\mathcal{C}$ be  the curve defined as the union of the conic-line fibers of the pencil. Assume that the singularities of $\cC$ at the base points of the pencil are transversal.
  Then the multiplicity $b_q$ of a primitive root of unity of order $q$ as a root of the Alexander polynomial of the curve $\cC$ is zero, unless $q =3$  or $q \mid m$. 
\end{lemma}

\begin{proof}
Immediately from \eqref{eq:Alexander_poly}, Theorem \ref{thm:global|local}  and Lemma \ref{lem:local_Alexander_poly}, since $\cC$ has only singularities that are either ordinary of multiplicity $2,3, m$ or of type $A_3$.
\end{proof}

\section{Pencils of conic-line curves}
\label{sec:CL_pencils}

The next classification result on 
pencils of conic-line curves with transversal intersections at the base points  is central to the proofs of our main theorems.

\begin{theorem} (\cite{CL})
\label{thm:CL_classification}
A primitive pencil $\mathcal{P}$ of degree $d \geq 3$ curves with transversal intersections at the base points has at most $m=6$ fibers which are conic-line arrangements. 
\end{theorem}

 Moreover, as a direct consequence of the results from \cite{CL} and their proofs (that involve estimating the Euler characteristic of the surface induced by the pencil by blowing up at the base points, in terms of  the Euler characteristics of the fibers; see for instance the proofs of \cite[Prop 3.2]{CL} and \cite[Thm 1.4]{CL}), one has, in such pencils, certain restrictions on the number of conic-line fibers which are pencils of lines. 
 This is stated explicitly in \cite{Sul}, as an outcome of a particular instance of the general techniques used in \cite{CL}.

\begin{proposition}
\label{prop:nr_pencil_of_lines}
In the hypothesis of Theorem \ref{thm:CL_classification},  if $n$ denotes the number of those fibers in the pencil $\mathcal{P}$  that are pencils of $d$ lines, then $n \leq 3$ and $n=1$ implies $m \leq 5$, $n=2$ implies $m \leq 4$, respectively  $n=3$ implies $m= 3$. 
\end{proposition}


Let us set some notations for the reminder of this paper. 
Let $\cC = \cC_1 \cup \cdots \cup \cC_m$ be a pencil-type conic-line arrangement seen as the union of the $m\geq 3$ conic-line fibers $ \cC_1,  \cdots, \cC_m$ of a pencil $\cP$ of degree $3$ conic-line arrangements and  $\cC_q =\cup_i C^i_q$ be the decomposition into irreducible components of the conic-line curves $\cC_q, \; q \in \overline{1,m}$.  
With these notations we have $A_{W_{\cC}}^1:=\bigoplus_{q = 1}^m (\bigoplus_i \sigma_q^i\Bbbk)$.\\
\noindent We call {\it base point (type) singularity} a singular point of $\cC$ that is a base point in the pencil $\cP$  and we call it {\it non-base point (type) singularity} otherwise.
With this preamble we are ready to prove the main result of the paper.\\

\subsection{ Proof of Theorem \ref{thm:Alex_poly_formulas}}
 Since the fibers are assumed to be reduced, each fiber is either a union of three lines (either a pencil of three lines or three lines in general position) or a union of a smooth conic and a line (that intersect either  transversely or the line is tangent to the smooth conic).
By Theorem \ref{thm:CL_classification} and Proposition \ref{prop:nr_pencil_of_lines}, $m \leq 6$ and $n\leq 3$. We will make a case by case discussion on the possible values of $m$ and $n$. 

Recall that, by the hypothesis, any point $P \in \Sing(\cC)$ is either an ordinary singularity  of multiplicity $2,3$ or $m$, or a simple singularity of type $A_3$. Moreover,  the $A_3$ type singularity can only appear as a non-base point type singularity, i.e. a singularity in one of the conic-line fibers. 

{\bf Case m = 6.}
If $m=6$, then none of  the conic-line members of the pencil is a pencil of lines, i.e. $n=0$. By Theorem \ref{thm:lower_bound}, $b_k \geq 4$ for $k \in \{2,3,6\}$.
We will show  that  $\beta_p \leq 4$ for $p \in \{2,3\}$ - for the case $p=2$ we will use the additional 'moreover' hypothesis. This will imply, by Theorem \ref{thm:upper_bound} (that is, inequality \eqref{eq:beta_ineq}), that $b_2=b_3=4$.
To compute $\beta_p$, by Proposition \ref{prop:beta_system} (see also Remark \ref{rem:beta_system}), one has to solve a linear system {\em S} of  equations  in $|\Irr(\cC)|$ variables  $(a_q^i), \; q \in \overline{1,6}, \; 1 \leq i \leq |\Irr(\cC_q)|$ over $\mathbb{Z}_p$. The system consists of the equations of type \eqref{eq:system} associated to each singularity of $\cC$ and equations of type \eqref{eq:p| d_i} for each irreducible component of degree divisible by $p$. This system translates the resonance condition $\omega \wedge \omega_1=0$, where $\omega = \sum_{q,i}a^i_q \sigma^i_q$.
 Each ordinary multiple point $P \in \Sing(\cC)$ produces either an equation of type \eqref{eq:sum} or a series of equations, of type  \eqref{eq:equalities}. A singularity $P$ of type $A_3$ produces a non-trivial equation in  {\em S},  of type \eqref{eq:system}, only if $p=3$.
 
Each singular point which is not a base point type is either an ordinary singular point of multiplicity $2$, or a singularity of type $A_3$. If the singular point is the transversal intersection of  the irreducible components $C_q^i, C_q^j$, then it gives rise to an equation in the linear system of type $a_q^i = a_q^j$. If the irreducible components $C_q^i, C_q^j$ intersect at an $A_3$ type singularity, then, only  if $p=3$, $a_q^i = a_q^j$.
 It follows that the coefficients of $\omega$ corresponding to the irreducible components of a given fiber are  equal, at least in the case $p=3$. 

A base point type singularity is an ordinary singularity of multiplicity $6$. So it produces an equation in the linear system of type \eqref{eq:sum}. In view of the fact that,  if $p=3$, $a_q:=a_q^i = a_q^j, \forall i,j$, for any fixed  $q$, all base point type singularities produce the same equation in the system {\em S} with $\mathbb{Z}_3$ coefficients, namely $a_1 + \cdots +a_6=0$. This means that, if $p=3$, the system has a solution depending on at most $5$ independent parameters, hence $\beta_3\leq 4$.

The case $p=2$ requires a more detailed treatment and the additional 'moreover' hypothesis. So, assume that there are at most two singularities $P, Q$ of type $A_3$ in $\Sing(\cC)$. These singularities are necessarily  non-base point type singular points. Assume without loss of generality that $P = C_1^1 \cap  C_1^2$ and  $Q = C_2^1 \cap  C_2^2$ are the possible singularities of type $A_3$.  Hence we  know that $a_q^i = a_q^j$ for all $q \geq 3, \; i \neq j$. Denote for convenience $a_q:=a_q^i = a_q^j, \; q \geq 3$. Consider the equations of type \eqref{eq:sum} determined by the singularities which are base point type. They are of type $a_1^r + a_2^s+a_3+a_4+a_5+a_6 =0$, where $1 \leq r, s \leq 2$ are arbitrary. It follows $a_1^1 = a_1^2$ and $a_2^1 = a_2^2$, hence $\beta_2\leq 4$.

 {\bf Case m = 5.}
If $m=5$, then  $n \in \{0,1\}$.  Non-base point type singularities are ordinary singularities of multiplicity $2,3$ or of type $A_3$ 
  and the base point type singularities are ordinary singularities of multiplicity $5$. 
  
Assume $p=5$. Just as in the previous case, the equation in {\em S} corresponding to non-base point type singularities imply $a_q:=a_q^i = a_q^j, \forall i \neq j$, for any fixed $q \in \overline{1,5}$. Then all the equations corresponding to the base point type singularities become  $a_1 + \cdots +a_5=0$.
 So $\beta_5 =3$.  Since, by Theorem \ref{thm:lower_bound}, $b_5 \geq 3$, we get $\beta_5 = b_5=3$.

Assume $p=3$. Any two irreducible components $C_r^i, C_s^j$ from two distinct conic-line fibers $\cC_r, \cC_s$ intersect at a base point type singularity, hence $a_r^i = a_s^j$, for any $r \neq s$ and any $i,j$. It follows the system {\em S} has a 1-dimensional solutions space, so $\beta_3 = 0$. Then, by Theorem \ref{thm:upper_bound} (i.e. inequality \eqref{eq:beta_ineq}), $b_3=0$.
  
 {\bf Case m = 4.}
 By Theorem \ref{thm:lower_bound},  $b_k \geq 2$ for $k \in \{2,4\}$.
 
 Let  $p=2$. If all non-base point type singularities are ordinary singularities, then we get, in the associated linear system {\em S}, equations of type \eqref{eq:equalities}. Hence for any fixed $q \in \overline{1,4}$ and any $i,j$, we have $a_q^i = a_q^j =:a_q$. The base point type singularities are ordinary singularities of multiplicity $4$, so they give rise to equations of type \eqref{eq:sum}, more precisely to the equation $a_1+a_2+a_3+a_4=0$. Hence  $\beta_2 \leq 2$.

 If there exist non-base point type singularities of type $A_3$, then, by hypothesis, we have at most two such singularities. Assume without loss of generality that $P = C_1^1 \cap  C_1^2$ and  $Q = C_2^1 \cap  C_2^2$ are possibly singularities of type $A_3$. In any case, we know that $a_q^i = a_q^j =: a_q$, for any $q \geq 3$ and any $i,j$.
We will show that  $a_q^i = a_q^j$, for $q=1,2$. To do this, we use the equations of the system {\em S} emerging from the base point type singularities.  These are $a_1^i+a_2^j+a_3+a_4=0$, for any $i,j$. This implies $a_1^i = a_1^j$ and $a_2^i = a_2^j$ so, again,  $\beta_2 \leq 2$. 
By inequality \eqref{eq:beta_ineq}, $b_k \leq \beta_2 \leq 2$ for $k \in \{2,4\}$. Since, by Theorem \ref{thm:lower_bound},  $b_k \geq 2$ for $k \in \{2,4\}$,  we obtain $b_2=b_4=2$. 
 
 If $p = 3$, we only need to take into account the equations in the system {\em S} emerging from the base point type singularities. They are all type \eqref{eq:equalities} equations, and we get $a_q^i=a_s^j$ for all $q \neq s \in \overline{1,4}$ and any $i, j$. 
 This implies that $\beta_3=0$, hence, by \eqref{eq:beta_ineq}, $b_3=0$.

 {\bf Case m = 3.}  
  If $m=3$, then $n \in \{0,1, 2, 3\}$. 
  The case $n=3$ is treated in \cite{DIM, PS3}. In this case $\cC$ is the Ceva arrangement, i.e. $\cC: (x^3-y^3)(x^3-z^3)(y^3-z^3)=0$ and $\beta_3 = b_3 =2$.
   It remains to consider the cases $n \in \{0,1,2\}$.  
  
  If $n=0$ then the non-base point type singularities of $\cC$ are singularities of type $A_3$ or ordinary multiple points of multiplicity $2$. The corresponding equations in the system {\em S} over $\mathbb{Z}_3$ imply $a_q:=a_q^i=a_q^j$ for any $i,j$ and any $q \in \overline{1,3}$. 
 The three parameters $a_1, a_2, a_3$ are connected by an equation of type  \eqref{eq:equalities},  $a_1+ a_2+a_3=0$, determined by any of the base point type singularities, which are ordinary singular points of multiplicity $3$. Hence $\beta_3 =1$.
  
  If $n=1$,  assume without loss of generality that the conic-line fiber $\cC_3$ is a pencil of lines. Then just as in the previous cases we get the identities  $a_1:=a_1^i, \; \forall i$ and $a_2:=a_2^i, \; \forall i$ involving the parameters of the irreducible components of the conic-line fibers $\cC_1, \cC_2$. The parameters corresponding to the three lines in $\cC_3$ are $a_3^1, a_3^2, a_3^3$, such that  $a_3^1+ a_3^2+ a_3^3=0$.
 The base point-type singularities give equations of type $a_1+a_2+a_3^i=0$, for any $i \in \{1,2,3\}$, hence $a_3^1 = a_3^2 =a_3^3$. We obtain $\beta_3=1$. 
  
 Let $n=2$. This means two of the conic-line fibers of the pencil, say  $\cC_1, \; \cC_2$, are pencils of lines. Then the third fiber $\cC_3$ cannot be the union of three lines in general position, see \cite{DIM}.  It follows that the third fiber $\cC_3$  is the union of a line and a smooth conic, which intersect either transversally or at an $A_3$ type singularity.
    In any case, we get an equality between the parameters corresponding to these two irreducible components, since they satisfy an equation of type \eqref{eq:system}. Denote then $a_3:=a_3^1 = a_3^2$. Since $\cC_1, \; \cC_2$ are pencils of lines, we have the equations $a_1^1+a_1^2+a_1^3=0$ and $a_2^1+a_2^2+a_2^3=0$.
  A line with parameter $a_1^i$ intersects the conic into two distinct points, which are base point type singularities. In particular, these points are ordinary singularities of multiplicity $3$. They give rise to equations of type \eqref{eq:sum}, more precisely, $a_1^i+ a_2^j+a_3 =0$ and  
  $a_1^i+ a_2^k+a_3 =0, \; k \neq j$. It follows $a_2^1=a_2^2=a_2^3$. By symmetry, we obtain $a_1^1=a_1^2=a_1^3$. This implies again $\beta_3=1$.
  
 By Theorem \ref{thm:lower_bound} and inequality \eqref{eq:beta_ineq}, $b_3=1$, in all subcases of the case $m=3$ for which $\beta_3=1$.\\
 
 Finally, the formulas of the respective Alexander polynomials from Theorem \ref{thm:Alex_poly_formulas} follow from the above computations and Lemma \ref{lem:b_q_trivial}.
$\Box$\\

\noindent {\bf Proof of Theorem \ref{thm:main}} Keeping in mind that $\beta_p(\cC)$ is a combinatorial invariant of the curve $\cC$,  Theorem \ref{thm:main}  follows immediately from the proof of Theorem \ref{thm:Alex_poly_formulas}, since the proof incidentally implies $b_p=\beta_p$ if $p$ is an odd prime, respectively $b_{2^s}=\beta_2$, $s =1,2$. $\Box$

\begin{remark}
In the proof of the above theorem, in the cases when all the irreducible components inside each fiber are in general position  (this implies that each fiber is completely $p$-reductive), the equality $\beta_p=m-2$ follows also directly from \cite[Theorem 32]{CMc}.
 We refer to \cite{CMc} for a definition of the notion of complete $p$-reductivness.
\end{remark}

\begin{remark}
\label{rem:p=3_strict_modular_ineq}
In all cases considered in the proof of Theorem \ref{thm:Alex_poly_formulas}, the modular inequality \eqref{eq:beta_ineq} turned out to be an equality. We already know that this is not the case in general, see \cite{CMc, Y} for examples of curves with $b_{2} < \beta_2(\cC)$. We add to this list an example of curve $\cC$  with $b_{3} < \beta_3(\cC)$. Let us revisit the Zariski pair of degree $6$ reduced and reducible curves introduced by Artal Bartolo in \cite{A}. Each is the union of a smooth cubic and three lines tangent to three inflection points of the cubic, and the lines are in general position.
 The two curves $\cC_1, \cC_2$ have the same combinatorial type, but distinct Alexander polynomials.
Let us denote by $b_3(\cC_i)$ the multiplicity of a primitive root of unity of order $3$ as a root of the Alexander polynomial of the curve $\cC_i, \;  i=1,2$. By \cite{A}, we know that $b_3(\cC_1)=1$ and $b_3(\cC_2)=0$. An easy computation using Proposition \ref{prop:beta_system} shows that $\beta_3(\cC_i) =\beta_3(W_{\cC_i}) =1$, so $b_3(\cC_2) < \beta_3(\cC_2)$.  
\end{remark}


\section{On the Alexander polynomial of a pencil-type curve with $m=4$ conic-line fibers}
\label{sec:m=4}

Throughout this section, unless otherwise specified, $\cC$ is a pencil-type conic-line arrangement defined by a pencil of conic-line arrangements of degree $3$ having $m=4$ conic-line fibers, denoted $\cC_1, \cdots, \cC_4$. 
As a  general rule, we denote the irreducible components of $\cC_i$ by $C_i^j$, i.e. $\cC_i = \cup_j C_i^j$. 
 However, when the irreducible decomposition of a fiber $\cC_i$ is known precisely, we will also use an alternative set of notations, for ease of use. Specifically, if $\cC_i$  is the union of three lines, we denote those lines by $L_i^1, L_i^2, L_i^3$, and if $\cC_i$ is the union of one line and one smooth conic, we denote by $L_i$ the line and by $C_i$ the smooth conic.

We will give a more precise formula for the Alexander polynomial of such a curve $\cC$, without imposing any restrictions on the type of singularities that are non-ordinary, as was the case in Theorem \ref{thm:Alex_poly_formulas}. Base point type singularities are assumed to be transversal, so any non-ordinary type singularity in $\Sing(\cC)$ must be the intersection of some irreducible components of a conic-line fiber $\cC_i$. This means that we can have at most four non-ordinary singularities, and they must be $A_3$ type singularities.  
Denote $$\cN:=\{P \in \Sing(\cC)\; |\; P\; \text{non-ordinary}\}.$$
Since Theorem  \ref{thm:Alex_poly_formulas} covers the case when there are  at most two non-ordinary singularities in $\Sing(\cC)$, we are left to deal with the cases $|\cN| \in \{3,4\}$.

Basically, given such a curve $\cC$, we will compute $\beta_2(\cC)$, in other words we will estimate the dimension of the space of solutions over $\mathbb{Z}_2$ of the linear system  {\em S} described in Remark \ref{rem:beta_system}. 
Via  inequality \eqref{eq:beta_ineq} and Theorem \ref{thm:lower_bound}, this will enable us to compute the exponents $b_2$ and $b_4$ from the formula \eqref{eq:Alexander_poly} of the Alexander polynomial of the curve $\cC$.

\begin{lemma}
\label{lem:4points&3conics}
$|C_q \cap  C_r \cap C_s| < 4$ for any $q \neq r \neq s$.
\end{lemma}

\begin{proof}
Assume the contrary, that there exist three conics  $C_q, C_r, C_s$ in $\Irr(\cC)$ that intersect at four base point type singularities. Consider the two  base point type singularities $L_q \cap C_r$. The irreducible component of the pencil member $\cC_s$ that passes through such a point must be in both instances the line $L_s$,  otherwise we would have $|C_r \cap C_s|>4$, contradiction. But this means that  $|L_q \cap L_s| > 1$, contradiction.
\end{proof}

\begin{proposition}
\label{prop:basepoint-4lines}
Assume $\cC$ has four lines and four smooth conics as irreducible components, i.e. $\cC_q = C_q \cup L_q, \; q \in \overline{1,4}$ .
 If  $L_1, L_2, L_3, L_4$ are concurrent lines, then $\beta_2(\cC)=2$. 
\end{proposition}

\begin{proof}
Since $L_1, L_2, L_3, L_4$ are concurrent lines, we have in the associated system {\it S} with $\mathbb{Z}_2$ coefficients the equation 
$$ (*) \; a_1^1+a_2^1+a_3^1+a_4^1=0.$$ 
Consider the base point type singularities $L_i \cap C_j$ and $L_j \cap C_i$, $i \neq j$. They are ordinary singularities  of multiplicity $4$ of the curve $\cC$, and necessarily the conics $C_k$ and $C_l$, $i \neq j \neq k \neq l$, pass through these points. The equations in  {\it S} corresponding to these points are
\begin{center}
$(**)$  $a_i^1+a_j^2+a_k^2+a_l^2=0$ and $a_i^2+a_j^1+a_k^2+a_l^2=0$. 
 \end{center}
 It follows $a_i^1+a_j^2 = a_i^2+a_j^1$ for any $i \neq j\in \overline{1,4}$, or equivalently $a^2_i-a^1_i = a^2_j-a^1_j$ for any $i \neq j \in \overline{1,4}$. Denote $a:=a^2_i-a^1_i$. By $(*)$ and $(**)$, $a=0$. This implies $a^2_i = a^1_i$ for all $i \in \overline{1,4}$. Hence there are $3$ independent parameters that define the solution space for the system {\it S}, which implies $\beta_2(\cC)=2$. 
\end{proof}

\begin{proposition}
\label{prop:m=4+1gen.pos.fiber}
Let $\cC$ be  such that $|\cN | = 3$.  Then  $\beta_2(\cC) = 2$. 
\end{proposition}

\begin{proof}
$|\cN | = 3$ means there are precisely three singularities of type $A_3$ in $\Sing(\cC)$, so we may assume that  $\cC_1$ has only ordinary singularities.
If $\cC_1$ is a union of three lines, the three lines intersect in either a triple point or in three double points. In any case, the equations in the system {\em S} corresponding to these singularities imply $a^1_1 = a_1^2 = a_1^3=:a_1$.  If $\cC_1$ is a union of a line and a smooth conic intersecting transversely, we get $a^1_1 = a_1^2 =:a_1$. 

Consider the four base point type singularities $C_2 \cap C_3$. By Lemma \ref{lem:4points&3conics}, $L_4$ passes through at least one of these points. On the other hand, since $|C_2 \cap L_4|=2$,  $L_4$ passes through at most two of these points. So there exist two singularities of type $C_1^i \cap C_2 \cap C_3 \cap C_4$ and $C_1^j \cap C_2 \cap C_3 \cap L_4$, which give rise to two equations in  {\em S}, $a_1+a^2_2+a_3^2+a_4^2=0$ and $a_1+a^2_2+a_3^2+a_4^1=0$. This implies $a_4^1=a_4^2=:a_4$.
The equations corresponding to a base point type quadruple point $C_2 \cap C_3$, respectively $L_2 \cap C_3$, become $a_1+a^2_2+a_3^2+a_4=0$ and $a_1+a_2^1+a_3^2+a_4=0$, hence $a_2^1 = a^2_2 =:a_2$. By the same type of argument $a_3^1 = a_3^2 =:a_3$. Then all the equations in the system {\em S} spell $a_1+a_2+a_3+a_4=0$. This means the space of solutions has three independent parameters, so $\beta_2(\cC)=2$.
\end{proof}

\begin{proposition}
\label{prop:4_A_3sing}
Let $\cC$ be  such that $|\cN | = 4$. 
 If there exists a base point type singularity which belongs to an odd number of lines from $\Irr(\cC)$, then $\beta_2(\cC)=2$. 
\end{proposition}

\begin{proof}
Consider first the $4$ base point type singularities at the intersection of the smooth conics $C_1$ and $C_2$. By Lemma \ref{lem:4points&3conics}, at least through one of the $4$ points $C_1 \cap C_2$ passes a line $L_i, \; i \in \{3,4\}$. Consider a point $C_1 \cap C_2 \cap L_3 \cap C_4^s$. $L_3$ can pass through at most $2$ of the  $4$ points $C_1 \cap C_2$. So there exists a base point type singularity $C_1 \cap C_2 \cap C_3 \cap C_4^l$. There are two cases to consider: $s=l$ and, respectively, $s \neq l$.

1. Case $s=l$.  Then the equations corresponding to the singularities $C_1 \cap C_2 \cap L_3 \cap C_4^s$ and $C_1 \cap C_2 \cap C_3 \cap C_4^l$ imply $a_3^2=a_3^1=:a_3$.
Consider now the $4$ base point type singularities  $C_1 \cap C_3$. Then, through at least one of those points passes the line $L_2$, again by Lemma \ref{lem:4points&3conics}. Hence there exists two base point type singularities of type 
 $C_1 \cap L_2 \cap C_3 \cap C_4^i$, respectively $C_1 \cap C_2 \cap C_3 \cap C_4^j$. 
 
 1.1 Case $i=j$. The equations in the system {\em S} corresponding to the singularities $C_1 \cap L_2 \cap C_3 \cap C_4^i$ and  $C_1 \cap C_2 \cap C_3 \cap C_4^j$ imply $a_2^1 = a^2_2 =:a_2$.
  The two equations from {\em S} corresponding to two base point type quadruple points $C_1 \cap C_4$ and $L_1 \cap C_4$ spell  $a^2_1+a_2+a_3+a_4^2=0$ and $a_1^1+a_2+a_3+a_4^2=0$, hence $a^1_1 = a_1^2 =:a_1$. By the same argument we get $a_4^1 = a_4^2 =:a_4$. Then all the equations in the system {\em S}  become $a_1+a_2+a_3+a_4=0$, so $\beta_2(\cC)=2$.
 
 1.2 Case $i \neq j$. Then $a^2_2 - a_2^1 = a_4^2 - a_4^1$. Consider the  $4$ base point type singularities at the intersection of the smooth conics $C_1$ and $C_4$. 
 By Lemma \ref{lem:4points&3conics}, the line $L_2$ passes through one of these points. Then we have two base point type singularities of type $C_1 \cap L_2 \cap C_3^u \cap C_4$ and $C_1 \cap C_2 \cap C_3^v \cap C_4$. The equations corresponding to these points imply $a_2^1 = a^2_2$ and we are now in the same setting as in the case 1.1. Hence $\beta_2(\cC)=2$. 

2. Case $s \neq l$. The equations corresponding to the singularities $C_1 \cap C_2 \cap L_3 \cap C_4^s$ and $C_1 \cap C_2 \cap C_3 \cap C_4^l$ imply $a_3^2-a_3^1 = a_4^2-a_4^1$. Actually we may assume that for any pair $C_i, C_j, i \neq j$ of smooth conics there is no pair of base point type singularities of type $C_i \cap C_j \cap C_q^u \cap C_r^s $ and $C_i \cap C_j  \cap C_q^v \cap C_r^s$, $u \neq v$, otherwise  we are back on case 1. 
This assumption implies the base point type singularities  $C_i \cap C_j$ are necessarily  of type $C_i \cap C_j \cap C_q^u \cap C_r^t$ and $C_i \cap C_j  \cap C_q^v \cap C_r^w$, $u \neq v, \; t \neq w$, hence we get
 the equalities $a_q^2 - a_q^1 = a_r^2 - a_r^1=:a$ for any $q \neq r$. 

\noindent By hypothesis, there exists a base point type singularity of $\cC$ through which passes an odd number of lines from $\Irr(\cC)$. The equation in {\em S} corresponding to this point can be written as $a^1_1+a_2^1+a_3^1+a_4^1+a=0$. 
 By Proposition \ref{prop:beta_system}, equation \eqref{eq:p| d_i}, we have $a^1_1+a_2^1+a_3^1+a_4^1=0$. 
It follows $a=0$, so again $a_q^2 = a_q^1$ for all $q \in \overline{1,4}$, which leads to $\beta_2(\cC)=2$. 
\end{proof}

\begin{theorem}
\label{thm:m=4_Alex_poly}
Let $\cC$ be such that either there exists a conic-line fiber $\cC_q$ that admits only ordinary singularities, or there exists a base point type singularity $P$ such that an odd number of lines from $\Irr(\cC)$ pass through $P$.
Then
$$\Delta_{\cC}(t) = (t-1)^{s-1}(t+1)^2(t^2+1)^2.$$
\end{theorem}

\begin{proof}
We already know from Theorem \ref{thm:Alex_poly_formulas} that $\Delta_{\cC}(t) = (t-1)^{s-1}(t+1)^{b_2}(t^2+1)^{b_4}.$ So we only need to show that $b_2=b_4=2$. Since, by Theorem \ref{thm:lower_bound}, we know that $b_2, b_4 \geq 2$,  it would be enough to show that $\beta_2 =2$ and then use inequality \eqref{eq:beta_ineq} to conclude $b_2= b_4= 2$.

Assume first that there exists a conic-line fiber $\cC_q$ which admits only ordinary singularities. This is equivalent to saying that  $|\cN| \leq 3$.
 If there are at most two singularities of type $A_3$ in $\Sing(\cC)$, the claim follows from the 'moreover' part of  Theorem \ref{thm:Alex_poly_formulas}.  If there are precisely $3$ singularities of type $A_3$ in $\Sing(\cC)$, it follows from Proposition \ref{prop:m=4+1gen.pos.fiber} that $\beta_2(\cC)=2$. 

Finally, assume that none of the conic-line fibers $\cC_q, \; q \in \overline{1,4}$, admit only ordinary singularities, in other words
$|\cN|=4$. 
 Then, by hypothesis, there exists a base point type singularity $P$ such that an odd number of lines from $\Irr(\cC)$ pass through it. As a consequence, by Proposition \ref{prop:4_A_3sing}, $\beta_2(\cC)=2$. 
\end{proof}

\subsection{An abstract weak combinatorial type}
\label{ss:ACC_example}
Notice that Proposition \ref{prop:4_A_3sing} treats  the case $|\cN|=4$, minus the situation when all base point type singularities involve an even number of lines. In this situation, it can happen (at least theoretically)  that $\beta_2(\cC)>2$. More precisely, we show next that there exists an 
'abstract' weak combinatorial type
 $W$ such that $\beta_2(W)=3$. 
 This will be an instance of the concept of {\it abstract curve combinatorics}, introduced by Cogolludo-Agust\'in \& Marco-Buzun\'ariz in \cite{BC}. This notion, in relation to curves, can be seen through analogy to the notion of matroid in relation to arrangements of hyperplanes. The intersection lattice of an arrangement of hyperplanes is in particular a matroid, but the class of matroids is strictly larger than the class of intersection lattices of hyperplane arrangements. A matroid that is isomorphic to the intersection lattice of a hyperplane arrangement is called {\it realizable}.
Likewise, to a curve one associates the weak combinatorial type, which is in particular an abstract curve combinatorics. 

Recall that the Combinatorial Aomoto complex defined in Section \ref{sec:preliminaries}  depends only  on the weak combinatorial type of the  curve $\cC$. So one can extend seamlessly this notion to an abstract weak combinatorial type.

Let us describe in detail the abstract weak combinatorial type $W$ mentioned earlier.  
A curve $\cC$ that would realize $W$  should be the union of four smooth conics $C_1, C_2, C_3, C_4$ and four lines $L_1, L_2, L_3, L_4$ such that a line $L_i$ and a smooth conic $C_i$ intersect in an $A_3$ type singularity and the rest of the singularities of $\cC$ would be these $9$ ordinary singularities of multiplicity $4$:
 \begin{enumerate}
\item $C_1 \cap C_2 \cap L_3 \cap L_4$
\item $C_1 \cap C_2 \cap C_3 \cap C_4$
\item $C_1 \cap C_2 \cap C_3 \cap C_4$
\item $C_1 \cap C_2 \cap C_3 \cap C_4$
\item $L_1 \cap L_2 \cap C_3 \cap C_4$
\item $C_1 \cap L_2 \cap L_3 \cap C_4$   
\item $C_1 \cap L_2 \cap C_3 \cap L_4$
\item $L_1 \cap C_2 \cap L_3 \cap C_4$ 
\item $L_1 \cap C_2 \cap C_3 \cap L_4$
 \end{enumerate}

 \begin{proposition}
 \label{prop:ACC}
A curve $\cC$ that realizes the abstract weak combinatorial type $W$ described above has  $\beta_2(\cC)=3$.
 \end{proposition}  
 
 \begin{proof}
 We need to compute the dimension of the space of solutions of the linear system {\em S} with $\mathbb{Z}_2$ coefficients.   In the notations from the beginning of this section, the variables, corresponding to the irreducible components of $\cC$, are $(a_q^i)_{1\leq q \leq 4, \; 1 \leq i \leq 2}$.
  There are $9$ equations of type \eqref{eq:system}, more precisely of type \eqref{eq:sum}, since they are induced by multiplicity $4$ ordinary singularities. These $9$ equations boil down to the following system 
  $$
\begin{cases}  
  a_q^2-a_q^1 =  a_s^2-a_s^1  \textrm{, for any } q,s \in \overline{1,4} \\
  a_1^1+a_2^1+a_3^1+a_4^1=0.
 \end{cases}
 $$
The equations of type \eqref{eq:p| d_i} spell $a_1^1+a_2^1+a_3^1+a_4^1=0$. So a solution to {\em S} depends on $4$ independent parameters, which implies $\beta_2(\cC)=3$.
\end{proof}

\begin{theorem}
\label{thm:uniqueACC}
Let $\cC$ be such that $|\mathcal{N}|=4$ and through each base point-type singularity  passes an even number of lines from $\Irr(\cC)$. Then the weak combinatorial type of $\cC$ is $W$.
\end{theorem}

\begin{proof}
Consider the four ordinary quadruple base point type singularities $C_1 \cap C_2$. By Lemma \ref{lem:4points&3conics}, $L_3$ passes through at least one of those four points.  $L_4$ must pass through the same point, by the hypothesis on the number of lines through base point singularities. We have  then the singular point  $C_1 \cap C_2 \cap L_3 \cap L_4$. By at least two of the points  $C_1 \cap C_2$ passes the conic $C_3$ (otherwise the the line $L_3$ would intersect the conics $C_1, C_2$ in more than 2 distinct points, contradiction). Again by the  hypothesis, those two points must be of type  $C_1 \cap C_2 \cap C_3 \cap C_4$. Since an even number of lines must pass through the fourth quadruple point  $C_1 \cap C_2$ as well, this point is either of type  $C_1 \cap C_2 \cap L_3 \cap L_4$ or of type $C_1 \cap C_2 \cap C_3 \cap C_4$.  But if it is of type $C_1 \cap C_2 \cap L_3 \cap L_4$ we would have two distinct points in the intersection $ L_3 \cap L_4$ , contradiction. So the fourth intersection point  $C_1 \cap C_2$ must be of type  $C_1 \cap C_2 \cap C_3 \cap C_4$.
To recap, up to now we have found one base point-type singularity of type $$C_1 \cap C_2 \cap L_3 \cap L_4$$ and three base point-type singularities of type $$C_1 \cap C_2 \cap C_3 \cap C_4$$
This inventory already contains three of the four distinct intersection points of type $C_1 \cap C_3$. By Lemma \ref{lem:4points&3conics}, the fourth singularity $C_1 \cap C_3$ is necessarily $C_1 \cap L_2 \cap C_3 \cap L_4$. Likewise, the fourth singularity $C_2 \cap C_3$ is necessarily $L_1 \cap C_2 \cap C_3 \cap L_4$,  the fourth singularity $C_2 \cap C_4$ is necessarily $L_1 \cap C_2 \cap L_3 \cap C_4$,   the fourth singularity $C_1 \cap C_4$ is necessarily $C_1 \cap L_2 \cap L_3 \cap C_4$ and  the fourth singularity $C_3 \cap C_4$ is necessarily $L_1 \cap L_2 \cap C_3 \cap C_4$.
This proves our claim.
\end{proof}

\begin{corollary}
\label{cor:W3_equiv}
Let $\cC$ be  such that $|\mathcal{N}|=4$
The following are equivalent:
\begin{enumerate}
\item  $\cC$ is such that through each base point-type singularity  passes an even number of lines from $\Irr(\cC)$
\item $W_{\cC} = W$
\item $\beta_2(\cC)=3$
\end{enumerate}
\end{corollary}

\begin{proof}
(1) $\Rightarrow$ (2) follows from Theorem \ref{thm:uniqueACC}, (2) $\Rightarrow$ (3) follows from Proposition \ref{prop:ACC} and (3) $\Rightarrow$ (1) follows from Proposition \ref{prop:4_A_3sing}.
\end{proof}
  
\begin{question}
Is the abstract weak combinatorial type $W$ realizable as the weak combinatorial type of a  complex plane projective curve? If so, do there exist two realizations as curves of $W$, such that their respective Alexander polynomials are distinct?
\end{question}

\begin{remark}
It is likely that the multiplicity of the primitive root of  order $2$ of the Alexander polynomial of a pencil-type conic-line arrangement $\cC$ that is the union of  $m=6$  fibers of a conic-line pencil of degree $3$ curves can be computed as well, even without the restrictive hypothesis on the non-base point type singularities from Theorem \ref{thm:Alex_poly_formulas}, by the same techniques used in the $m=4$ case, but the arguments would be more elaborate. We leave that to the interested reader.
\end{remark}

\end{document}